\documentclass[11pt,a4paper]{amsart}
\usepackage{amsmath,amssymb,amsfonts, float}
\usepackage[all]{xy}
\usepackage{caption}
\usepackage{enumerate}
\usepackage{a4wide}

\usepackage{bm}
\usepackage{graphicx}
\usepackage[breaklinks=true]{hyperref}

\usepackage{xcolor}
\usepackage{multicol}
\usepackage{tabularx}

\usepackage{hyperref}
\usepackage[capitalize]{cleveref}

\usepackage[normalem]{ulem}

\newtheorem{thm}{Theorem}[section] 
\newtheorem*{thm*}{Theorem} 
\newtheorem{prop}[thm]{Proposition}
\newtheorem{lem}[thm]{Lemma}
\newtheorem{cor}[thm]{Corollary}

\theoremstyle{definition}
\newtheorem{definition}[thm]{Definition}
\newtheorem{expl}[thm]{Example}

\newtheorem{question}[thm]{Question}
\newtheorem{rem}[thm]{Remark}

\DeclareMathOperator{\C}{\mathbb{C}}
\DeclareMathOperator{\Z}{\mathbb{Z}}
\DeclareMathOperator{\N}{\mathbb{N}}
\DeclareMathOperator{\F}{\mathbb{F}}

    \DeclareFontFamily{U}{wncy}{}
    \DeclareFontShape{U}{wncy}{m}{n}{<->wncyr10}{}
    \DeclareSymbolFont{mcy}{U}{wncy}{m}{n}
    \DeclareMathSymbol{\Sha}{\mathord}{mcy}{"58}

\numberwithin{equation}{section}

\renewcommand{\i}{\mathrm{i}}

\DeclareSymbolFont{bbold}{U}{bbold}{m}{n}
\DeclareSymbolFontAlphabet{\mathbbold}{bbold}

\begin{document}
\title{On certain properties of the \\ $p$-unitary Cayley graph  over a finite ring}
\date{}
\dedicatory{Dedicated to Professor J\'an Min\'a\v{c} on the occasion of his 71th birthday }
 \author{Tung T. Nguyen, Nguy$\tilde{\text{\^{e}}}$n Duy T\^{a}n }

 \address{Department of Kinesiology, Western University, London, Ontario, Canada N6A 5B7}
 \email{tungnt@uchicago.edu}
 
  \address{
 Falculty of   Mathematics and 	Informatics, Hanoi University of Science and Technology, 1 Dai Co Viet Road, Hanoi, Vietnam } 
\email{tan.nguyenduy@hust.edu.vn}
 
\maketitle
\begin{abstract}
    In recent work \cite{chudnovsky2024prime}, we study certain Cayley graphs associated with a finite commutative ring and their multiplicative subgroups. Among various results that we prove, we provide the necessary and sufficient conditions for such a Cayley graph to be prime. In this paper, we continue this line of research. Specifically, we investigate some basic properties of certain $p$-unitary Cayeley graphs associated with a finite commutative ring. In particular, under some mild conditions, we provide the necessary and sufficient conditions for this graph to be prime. 
\end{abstract}

\section{Introduction}
Let $G$ be an undirected graph. A homogeneous set in  $G$ is a set $X$ of vertices of $G$ such that every vertex in $V(G) \setminus X$ is adjacent to either all or none of the vertices in $X$. A homogenous set $X$ is said to be non-trivial if $2 \leq X < |V(G)|$. As explained in  \cite[Section 1.1]{chudnovsky2024prime}, the existence of non-trivial homogeneous sets allows us to decompose $G$ as a joined union of smaller graphs. Such a decomposition is important for many problems in network theory and dynamical systems on them (see \cite{boccaletti2014structure, jain2023composed, kivela2014multilayer, nguyen2022broadcasting}). In \cite[Section 4]{chudnovsky2024prime}, we study the prime property of various Cayley graphs associated with a finite commutative ring. There, we provide necessary and sufficient conditions for the existence of homogeneous sets in those Cayley graphs under some mild conditions (see \cite[Theorem 4.1]{chudnovsky2024prime}). In this paper, we continue and expand this line of research to some other directions. Specifically, we will investigate some basic properties of the  $p$-unitary Cayley graphs whose definition we will explain below.

Let $R$ be a finite commutative ring and $p$ a natural number.  Let $S = (R^{\times})^p$ be the set of all invertible $p$-th powers in $R$. Since we only deal with undirected graphs in this paper; we will assume that $-1 \in (R^{\times})^p$. 
\begin{definition}(See  \cite[Section 1.2]{podesta2023waring})
The $p$-unitary Cayley graph $G_R(p)$ is the undirected graph with the following data 

\begin{enumerate}
    \item The vertex set $V(G_R(p))$ of $G_R(p)$ is $R.$ 
    \item Two vertices $a,b \in V(G_R(p))$ are connected if and only $a-b \in (R^{\times})^p.$
\end{enumerate}
\end{definition}
\begin{rem}
    Because we want to deal with one prime at a time, we will assume that $p$ is a prime number throughout the rest of this article.
\end{rem}

\begin{expl}
The Cayley graph $X_{3, \F_{13}}$ is described by \cref{fig:p_13_3}. It is a connected, regular graph of degree $4.$
\begin{figure}[H]
\centering
\includegraphics[width=0.4 \linewidth]{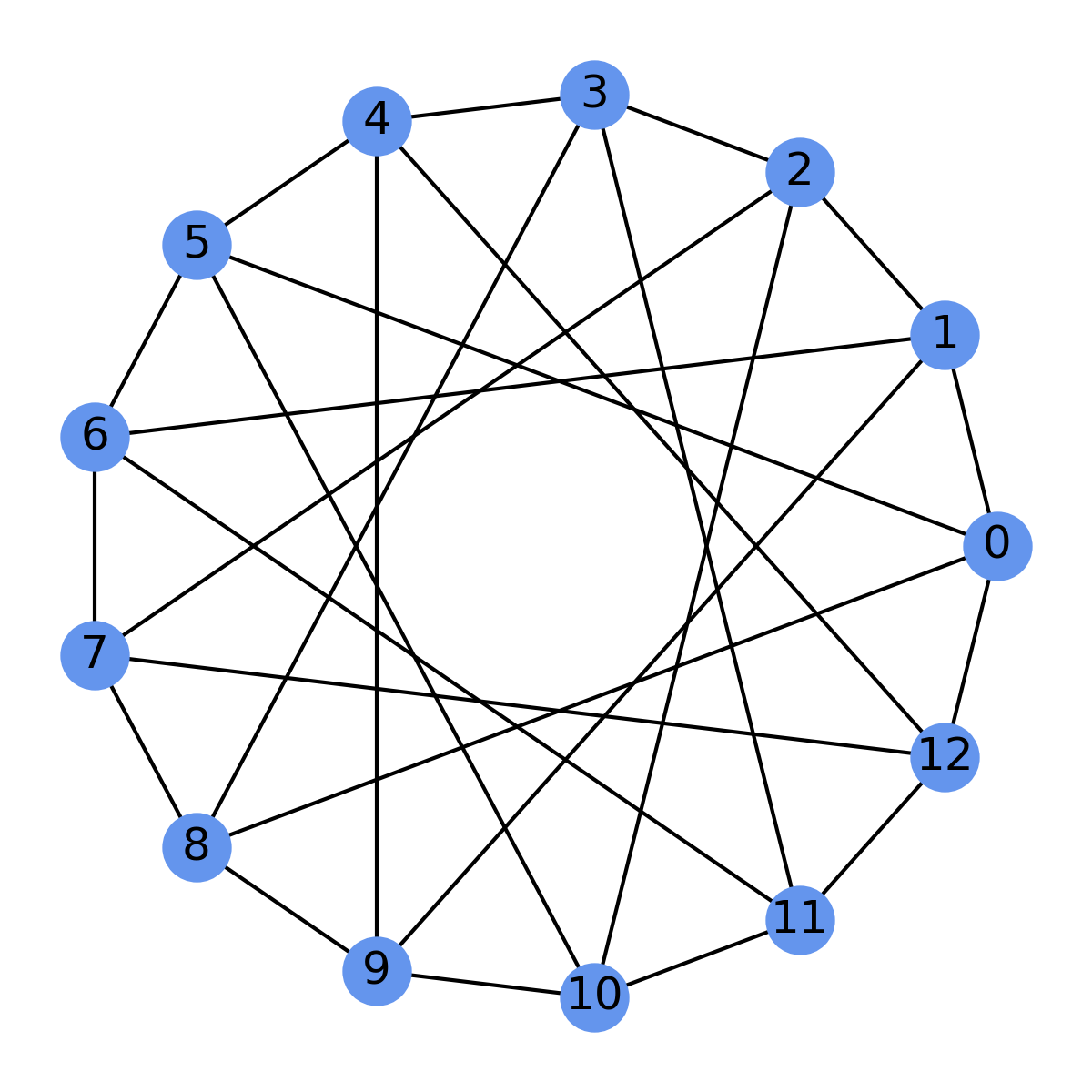}
\caption{The Cayley graph $G_{\F_{13}}(3)$}
\label{fig:p_13_3}
\end{figure}
\end{expl}

The case $p=1$, where $G_{R}(1)$ is often referred to as the unitary Cayley graph associated with $R$, is studied extensively in the literature (see for example \cite{unitary, bavsic2015polynomials, chudnovsky2024prime, klotz2007some, su2016diameter}). When $R$ is a finite field of characteristics $\ell \neq p$, $G_R(p)$ is called a generalized Paley graph. These generalized Paley graphs have interesting arithmetic and spectral properties and they have found various applications in coding and cryptography theory (see \cite{ghinelli2011codes, javelle2014cryptographie, podesta2019spectral, podesta2023waring}).

In \cite{chudnovsky2024prime}, amongst various results that we discover, we are able to classify all prime unitary Cayley graphs (see \cite[Theorem 4.35]{chudnovsky2024prime}). In light of this theorem, the following question seems to be quite natural. 

\begin{question} \label{question:prime}
    When is $G_R(p)$ a prime graph?
\end{question}
By definition, a connected component of $G$ (or of its complement $G^c$) is necessarily a homogeneous set. Consequently, if $G_R(p)$ is prime, then it must be connected and anticonnected unless $|V(G)|=2$ (recall that a graph is called anticonnected if $G^c$ is connected). As a result, a closely related question about $G_R(p)$ is the following. 

\begin{question} \label{question:connected}
    When is $G_R(p)$ connected and anticonnected?
\end{question}
By  \cite[Corollary 2.5]{podesta2023waring}, we know the complete answer for \cref{question:connected} when $R$ is a local ring such that $p$ is invertible in $R.$ The cases where either $p$ is not invertible in $R$ or $R$ is not a local ring are more complicated and we will deal with them in this work. 

\subsection{Outline}
In \cref{sec:background}, we recall some backgrounds in graph theory as well as some relevant results in our previous work \cite{chudnovsky2024prime}. To proceed further, we remark that by the structure theorem, $R = R_1 \times R_2 \times \ldots \times R_d$ where $R_i$'s are finite local rings. We then see that  $G_R(p)$ is the tensor product of $G_{R_i}(p)$ (see \cref{def:tensor_product} for the definition of the tensor product of graphs). More specifically
$$ G_R(p) \cong \prod_{i=1}^d G_{R_i}(p).$$
From this tensor product decomposition, it seems natural to first study the case where $R$ is a finite local ring. In this case, as expected, the behavior of $G_R(p)$ depends on whether $p$ is invertible on $R$. In \cref{sec:local_ring_l_p}, we discuss the case where $R$ is local and $p$ is invertible in $R$. Using the results in \cite{podesta2021_finitefield}, we are able to provide a complete answer to \cref{question:prime} (see \cref{thm:prime_l_p}). In this section, we also study induced subgraphs of $G_R(p)$ which we need later for the general case. However, this topic could be of independent interest. \cref{sec:local_ring_p_p} deals with the case where $R$ is local but $p$ is not invertible in $R.$ In this case, using some rather involved ring-theoretic arguments, we are also able to give a complete answer to \cref{question:prime} (see \cref{thm:prime_local_p_p}). Additionally, we also investigate some induced subgraphs of $G_R(p)$ in this case. As a by-product, we introduce some auxiliary polynomials $f_p, g_p$ that possess some interesting arithmetic properties.  Finally in \cref{sec:general_case}, we study \cref{question:prime} and \cref{question:connected} in the general case. Here, we provide the necessary and sufficient conditions for $G_R(p)$ to be prime (see \cref{thm:prime_general_case}). We also discuss some special cases where we can verify these conditions directly.  

\section{Backgrounds and previous work} \label{sec:background}
In this section, we discuss some fundamental concepts in graph theory. We also recall some results in \cite{chudnovsky2024prime} that we will use throughout this article. 
\begin{definition}[Induced subgraph]
  Let $G$ be a graph and $H \subseteq V(G )$ a non-empty subset. The induced subgraph $\Gamma[H]$ on $H$ is the subgraph of $\Gamma$ with the following data 
  \begin{enumerate}
      \item $V(G[H]) = H$,
      \item $E(G[H]) = \{ (x, y) \in E(G) | x, y \in H\}$.
      \end{enumerate}
\end{definition}
\begin{definition}[Tensor product of graphs] \label{def:tensor_product}
    Let $G,H$ be two graphs. The tensor product $G\times H$ of $G$ and $H$ (also known as the direct product) is the graph with the following data:
    \begin{enumerate}
    \item The vertex set of $G \times H$ is the Cartesian product $V(G) \times V(H)$, 
    \item Two vertices $(g,h)$ and $(g',h')$ are connected in $G \times H$ if and only if $(g,g') \in E(G)$ and $(h, h') \in E(H).$
\end{enumerate}
\end{definition}

\begin{definition}[Wreath product] \label{def:wreath_product}
  Let $\Gamma,\Delta$ be two graphs. We define the wreath product  of $\Gamma $ and $ \Delta$ as the graph $\Gamma \cdot \Delta$  with the following data 
  \begin{enumerate}
      \item The vertex set of $\Gamma \cdot \Delta$ is the Cartesian product $V(\Gamma ) \times V(\Delta )$,
      \item $(x,y)$ and $(x',y')$ are connected in $\Gamma \cdot \Delta$ if either $(x,x') \in E(\Gamma)$ or $x=x'$ and $(y,y') \in E(\Delta)$. 
  \end{enumerate}
\end{definition}

\begin{definition}[The complete graph $K_n$]
$K_n$ is the graph on $n$ vertices which are pairwisely connected. 

\end{definition}
We also recall the definition of certain Cayley graphs associated with a ring as discussed in \cite[Section 4]{chudnovsky2024prime}.
\begin{definition}
    Let $R$ be a commutative ring, and $S$ a subgroup of the set $R^{\times}$ of units of $R$ such that $-1 \in S$. We denote ${\rm Cay}(R, S)$ the Cayley graph $\text{Cay}((R,+), S)$. To be more specific, the vertex set of ${\rm Cay}(R,S)$ is $R$ and two vertices $a, b \in R$ are connected if and only if $a-b \in S.$
    \end{definition}

To study the prime property of ${\rm Cay}(R,S)$, we will make use of the following result which was first discovered in \cite{chudnovsky2024prime}. 

\begin{prop} (see \cite[Theorem 4.1]{chudnovsky2024prime} and \cite[Proposition 4.7]{chudnovsky2024prime}) \label{prop:prime_old}
    Suppose that ${\rm Cay}(R, S)$ is connected and anti-connected. If ${\rm Cay}(R,S)$ is not prime, then there exists a non-trivial ideal $I$ such that $I$ is a homogeneous set. Furthermore, $I$ is a subset of the Jacobson radical of $R.$
\end{prop}

\section{$G_R(p)$ where $R$ is a finite local ring of residue characteristics $\ell \neq p$} \label{sec:local_ring_l_p}

\subsection{When is $G_R(p)$ prime} 
Let $R$ be a local ring and $M$ its maximal ideal. Let $k=R/M$ be the residue field. We suppose that ${\rm char}(k)\not=p$.
\begin{prop} \label{pro:l_p_field}
$M$ is a homogeneous set in $G_R(p).$ As a result, if $G_R(p)$ is prime, then $R$ is a field.
\end{prop}
\begin{proof} Since $R$ is finite, $R$ is henselian. Hence all roots of $x^p-a$ over $k$ lift to a root over $R$ by Hensel's lemma. Thus $(R^\times)^p+M=(R^\times)^p$ and $M$ is homogeneous.
\end{proof}
Since the Jacobson radical of a field is $0$, by \cref{prop:prime_old}, we know that the question of whether $G_{R}(p)$ is prime reduces to the question of whether it is connected and anticonnected. We remark that the graph $G_{R}(p)$ is not always connected (\cref{fig:p_16_5} shows an example of a $p$-unitary Cayley graph with $4$ connected components). In fact, in \cite[Corollary 2.5]{podesta2023waring}, the authors provide the precise condition for $G_{R}(p)$ to be connected. To recap this result, we need to recall the definition of a primitive divisor.

\begin{definition}
    Let $\ell$ be a prime number and $m,n$ two natural numbers. We say that $n$ is a primitive divisor of $\ell^m-1$ if $n| \ell^m-1$ and $n \nmid \ell^a-1$ for each $1 \leq a \leq m-1.$ In this case, we write $n \dag \ell^m-1.$
\end{definition}

\begin{prop} (\cite[Corollary 2.5]{podesta2023waring} and \cite[Corollary 3.1]{podesta2021_finitefield}) \label{prop:l_p_connected}
    Suppose that $R=k$ is a field of order $\ell^m$ where $\ell \neq p.$ Then $G_{R}(p)$ is connected if and only one of the following conditions hold 
    
    \begin{enumerate}
        \item $p \nmid \ell^m-1.$ In this case $G_R(p) = K_{\ell^m}, $
        \item $p| \ell^m-1$ and $\frac{\ell^m-1}{p} \dag \ell^m-1.$
    \end{enumerate}
\end{prop}

\begin{figure}[H]
\centering
\includegraphics[width=0.4 \linewidth]{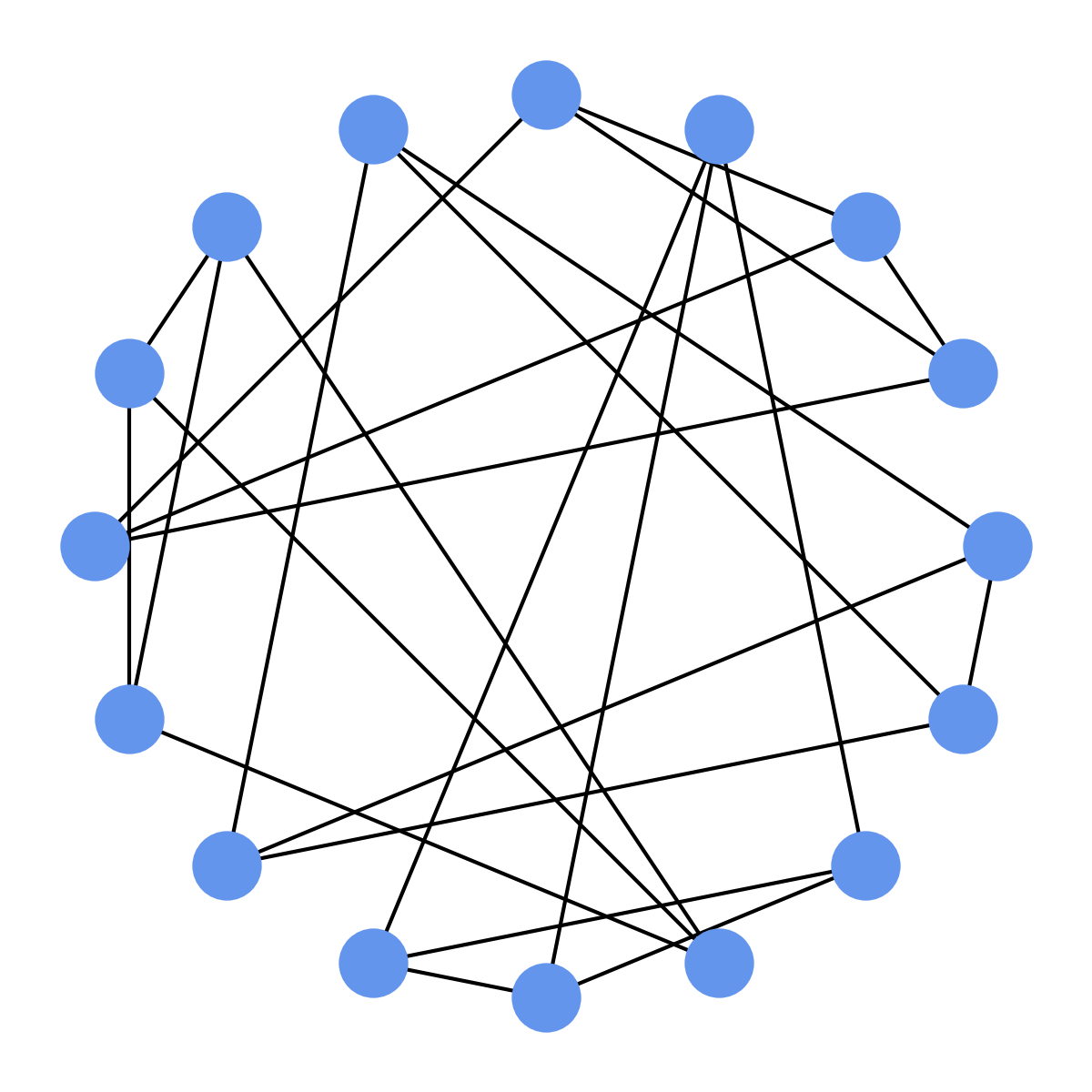}
\caption{The Cayley graph $G_{\F_{16}}(5)$}
\label{fig:p_16_5}
\end{figure}
\begin{cor} \label{cor:l>p}
    Suppose that $R=k$ is a field of order $\ell^m$ where $\ell \neq p.$ If $\ell >p$ then $G_R(p)$ is connected.
\end{cor}

\begin{proof}
 If $p \nmid \ell^m-1$, then the statement is true since $G_R(p)= K_{\ell^m}$. Now, suppose that $p| \ell^m-1.$  If $1 \leq a \leq m-1$, then 
    \[ \ell^a -1 \leq \ell^{m-1}-1 < \frac{\ell^m-1}{\ell} < \frac{\ell^m-1}{p}.\]
    Consequently $\dfrac{\ell^m-1}{p} \dag \ell^m-1$ and hence $G_{R}(p)$ is connected. 
\end{proof}
\begin{prop} \label{prop:l_p_anti}
    Suppose that $R=k$ is a field of order $\ell^m$ where $\ell \neq p$. Assume further that $G_R(p)$ is connected. The following conditions are equivalent.
\begin{enumerate}
    \item $p \nmid \ell^m-1.$
    \item $(R^{\times})^p = R^{\times} = R \setminus{0}$. 
    \item $G_R(p) = K_{\ell^m}.$
    \item $G_R(p)$ is not anticonnected. 
\end{enumerate}
\end{prop}

\begin{proof}
    The equivalence of $(1)-(2)-(3)$ follows from the fact that $R^{\times}$ is a cyclic group of order $\ell^m-1.$ Clearly, $(3)$ implies $(4).$ We claim that $(4)$ implies $(2)$ as well. Suppose, in fact, that $G_R(p)$ is not anticonnected but $R^{\times} \neq (R^{\times})^p.$ Let $a \in R^{\times} \setminus (R^{\times})^p.$ Let $\Phi_a: R \to R$ be the multiplication by $a$ map. Since $a \not \in (R^{\times})^p$, under $\Phi_a$, $G_R(p)$ is a subgraph of $G_R(p)^c.$ Since $G_R(p)$ is connected, $G_R(p)^c$ is connected as well. This is a contradiction. We conclude that $(R^{\times})^p = R^{\times}$. 
\end{proof}

Combining \cref{pro:l_p_field}, \cref{prop:l_p_connected} and \cref{prop:l_p_anti} we answer \cref{question:prime} in the case $R$ is a finite local ring with residue characteristics $\ell \neq p.$

\begin{thm} \label{thm:prime_l_p}
    Suppose that $R=k$ is a finite local ring of residue characteristics $\ell \neq p$. Then $G_{R}(p)$ is prime if and only if the following conditions hold 
    \begin{enumerate}
        \item $R$ is a finite field of order $\ell^m.$
        \item $p|\ell^m-1$ and $\frac{\ell^m-1}{p} \dag \ell^m-1.$  
    \end{enumerate}
\end{thm}

\subsection{Induced subgraphs of $G_R(p)$} \label{sec:induced_l_p}
In this section, we study various induced subgraphs of $G_{R}(p).$ This will be helpful later on when we study the general case.  First, we remark that by \cref{pro:l_p_field}, we can safely assume that $R$ is a finite field of characteristics $\ell \neq p.$  Therefore, we will assume that $R$ is a finite field throughout this section. Furthermore, if $p \nmid |R|-1$ then $G_{p}(R) = K_{|R|}$. Consequently, most of our results would be either obvious or in need of some easy modifications.  Therefore, we will also assume that $p$ is a divisor of $|R|-1.$

Our main tool in this section is the theory of character sums and their Weil bounds (see \cite{andre1948courbes, mauduit1997finite}). To do so, we first make a connection between $G_{R}(p)$ and character theory. We know that $R^{\times}$ is a cyclic group of order $|R|-1.$ By fixing a generator 
$g \in R^{\times}$, we have an embedding 
\[ \iota: R^{\times} \hookrightarrow \C^{\times},\]
sending $g \mapsto \zeta_{|R|-1}$ where $\zeta_{|R|-1}$ is a primitive $(|R|-1)$-th root of unity. Let $\chi: R^{\times} \to \C$ be the character defined by $\chi(g)= \zeta_{|R|-1}^{\frac{|R|-1}{p}}.$ Then $\chi$ is a character of order $p$. Furthermore, $\chi(x)=1$ if and only $x \in (R^{\times})^p.$ We also define $\chi(0)=0$.

\begin{rem}
We recall that in \cite{chudnovsky2024prime, paleygraph}, we define and study the Paley graph $P_{\chi}$ as ${\rm Cay}(R, \ker(\chi))$. The above discussion shows that $P_{\chi} = G_{R}(p).$ 
\end{rem}

To apply the Weil bound, we introduce the following auxiliary polynomials. 

\[ P_1(x) = \frac{1}{p} \left(\frac{1-x^p}{1-x} \right) =\frac{1}{p} \left(1+x+\ldots+x^{p-1} \right),  \]
and 
\[ P_0(x) = 1-P_1(x) = \frac{p-1}{p} - \frac{1}{p} \left(x+x^2+\ldots+x^{p-1} \right).\]
We can see that if $z$ is a $p$-th root of unity; i.e. $z^p=1$, then 

\[
 P_1(z) = 
\begin{cases}
    1 & \text{if } z=1 \\
    0 & \text{else.} 
\end{cases}
\]
Similarly

\[
 P_0(z) = 
\begin{cases}
    0 & \text{if } z=1 \\
    1 & \text{else.} 
\end{cases}
\]

\begin{lem} \label{lem:K_3}
    Suppose that $R=k$ is a field of characteristics $\ell \neq p$. Suppose that $|R| \geq (p+1)^4$. Then $K_3$ is a subgraph of $G_R(p)$. More specifically, there exists $a \in R \setminus \{0, 1 \}$  such that the induced subgraph on $\{0, 1, a \}$ is $3.$
\end{lem}
\begin{proof}
    We will look for $a$ of the form $a=x^p$ where $x \in R^{\times}.$ With this choice, we only need to make sure that $(1, x^p) \in E(G_R(p)).$ In other words, we need to find $x \neq 0$ such that $\chi(1-x^p)=1.$ Let $T$ be the set of all $a \in R$ such that either $a=0$ or $a^p=1.$ We know that $|T|=p+1.$ Let $S$ be the set of all $x \in R^{\times}$ such that $\chi(1-x^p)=1.$
    Let us consider the following function 
\[ f(a) = P_{1}(\chi(1-a^p)) =\frac{1}{p} \frac{\chi(1-a^p)^p-1}{\chi(1-a^p)-1} = \frac{1}{p} \sum_{k=0}^{p-1} \chi^k(1-a^p). \]
We can see that if $a \in R$ then $0 \leq |f(a)| \leq 1$. Additionally, if $a \in T \setminus \{0\}$ then $f(a)= \frac{1}{p}$; $f(0)=1$. Furthermore, if $x \not \in T$, then $f(a)=1$ if $a \in S$ and $f(a)=0$ otherwise. We conclude that 

\[ |S| = \sum_{a \in R} f(a) -  \sum_{a \in T} f(a) = \frac{1}{p} \sum_{k=0}^{p-1} \left(\sum_{a \in R} \chi^k(1-a^p) \right) -\frac{p+p}{p}.\]

By the Weil bound, we know that for $1 \leq k \leq p-1$
\[ \left|\sum_{a \in R} \chi^k(1-a^p) \right| \leq (p-1) \sqrt{|R|} .\] 
On the other hand, when $k=0$ we have 
\[ \left|\sum_{a \in R} 1 \right| = |R|. \]
Therefore, by triangle inequality we have 
\[ p|S| \geq |R| -(p-1)^2 \sqrt{|R|}-2p. \]
By an elementary calculation, we can see that that if $|R| \geq (p+1)^4$ then $|S|>0.$ In other words, we can find $x \in (R^{\times})$ such that the induced graph on $\{0, 1, x^p \}$ is $K_3.$
\end{proof}

\begin{rem}
   The bound $|R|>(p+1)^4$ also implies that the Waring problem for $R$ has an exact answer (see \cite[Example (a), Page 3]{podesta2021_finitefield}.) 
\end{rem}

\begin{cor} \label{cor:contains_cycle}
Suppose that $R=k$ is a field of characteristics $\ell \neq p$. Assume that either $\ell$ is odd or $|R|>(p+1)^4$. Then $R_{R}(p)$ is a not bipartite graph. 
\end{cor}

\begin{proof}
     $G_R(p)$ contains the $C_\ell$-cycle 
     \[ 0 \to 1 \to \cdots \to \ell-1 \to 0. \]
     Therefore, if $\ell$ is odd then $G_{R}(p)$ is not a bipartite graph. Similarly, if $|R|>(p+1)^4$ then $\textbf{C}_3=K_3$ is an induced subgraph of $G_{R}(p)$. Therefore, $G_R(p)$ is not a bipartite graph.
\end{proof}

\begin{rem}
    We wrote some Sagemath code to search for an example where $G_R(p)$ is bipartite (by \cref{cor:contains_cycle}, we only need to consider the case $R$ is a finite field of characteristics $2$). So far, our attempt has been unsuccessful. This leads us to wonder whether such an example exists at all. 
\end{rem}

In \cite{cohen1988clique}, the authors show that for each $n$, the complete graph $K_n$ is an induced subgraph of the generalized Paley graph $G_{\F_{\ell}}(2)$ as long as $\ell$ is large enough. More generally, the main result in \cite{bollobas1981graphs} shows that for a fixed graph $G$, $G$ is an induced subgraph of $G_{\F_{\ell}}(2)$ if $\ell$ is big enough.  At the time of our writing, it is unclear to us whether similar results have been obtained for $G_{\F_{\ell}}(p).$ Since the proof for this fact is quite standard and straightforward, we provide it here for the sake of completeness. 
\begin{thm} \label{thm:induced_graph}
   Let $G$ be an undirected graph. Let $p$ be a fixed prime number and $R$ a prime finite field of characteristics $\ell \neq p$; i.e, $R=\F_{\ell}$. Assume further that $p|\ell-1.$ Then there exists a constant $C$ depending on $p$ and $G$ such that if $\ell>C$ then $G$ is isomorphic to an induced subgraph of $G_R(p).$
\end{thm}

\begin{proof}
  Let $k=|G|.$  The key idea in our proof is to find $y \in \F_{\ell}$ and a tuple $(a_1, a_2, \ldots, a_k) \in \Z^k$ such that the induced graph on $\{y^{a_1}, y^{a_2}, \ldots, y^{a_k} \}$ is $G$.  We can choose $(a_1, a_2, \ldots, a_k)$ in such a way that $a_i-a_j$ are pairwisely different. For example, we can take $a_i = 2^{i-1}.$ In order to make sure that the induced graph on $\{y^{a_1}, y^{a_2}, \ldots, y^{a_k} \}$ is $G$, it is sufficient to find $y \in \F_{\ell}$ such that $\chi(y)=1$ and 

\[
\chi(y^{a_j-a_i}-1)  =
\begin{cases}
    1 & \text{if } (i,j) \in V(G) \\
    \neq 1 & \text{else.} 
\end{cases}
\]

Let $S$ be the set of all such $y$. Our goal is to show that $|S|>0$ whenever $\ell$ is sufficiently large. Similar to the proof of \cref{lem:K_3}, we will do so by a counting argument.  Let $C= (c_{ij})$ be the adjacency matrix of $G.$ Define the following function  
\[ f(x) = P_1(\chi(x)) \prod_{i<j} P_{c_{ij}} (\chi(x^{a_j-a_i}-1)).\]
Let $T$ be the set of $y$ such that either $y =0$ or $y$ is a root of the equation $y^{a_j-a_i}-1$ for some $i<j.$ Then $T$ is a finite set. For each $x \in \F_{\ell}$, $0 \leq |f(x)| \leq 1.$ Furthermore, if $x \in \F_{\ell} \setminus T$, then $f(x)=1$ if $x \in S$ and $f(x)=0$ otherwise. 
Therefore, we have 
\[ |S| = |\sum_{y \in \F_\ell} f(y) - \sum_{y \in T} f(y)| \geq |\sum_{y \in \F_\ell} f(y)| - |T|. \]
By the Weil bound, we also have 
\[ |\sum_{y \in \F_\ell} f(y)| \geq \frac{(p-1)^N}{p^{{|G| \choose 2}+1}} \ell - C \sqrt{\ell},\]
where $N$ is the number of $(i,j)$ such that $i<j$ and $(i,j) \not \in V(G)$ and $C$ is a constant depending on $G$ and $p$ only. We conclude that 

\[ |S| \geq \frac{(p-1)^N}{p^{{|G| \choose 2}+1}} \ell - C \sqrt{\ell} - |T| .\] 
Therefore, for $\ell$ big enough, $|S|>0.$ In other words, we can find $y$ such that the induced graph on $\{y^{a_1}, y^{a_2}, \ldots, y^{a_k} \}$ is $G.$
\end{proof}
\begin{rem}
In \cite[Theorem 9.1]{unitary} the authors classify all $G_{R}(1)$ which are perfect. In particular, they show that the graph $G_{R}(1)$ is always a perfect graph if $R$ is a finite local ring. However, things are different when $p>1$. In fact, by \cref{thm:induced_graph}, for each $p$ there exists a constant $C_p$ such that if $\ell >C_p$, then the cycle graph $\textbf{C}_5$ is an induced subgraph of $G_{\F_\ell}(p)$. Consequently $G_{\F_\ell}(p)$ is not perfect. 
\end{rem}

\section{ $G_R(p)$ where $p$ is a local ring of residue characteristics $p$} \label{sec:local_ring_p_p}
In this section, we study \cref{question:prime} and \cref{question:connected} in the case $R$ is a local field of residue characteristics $p.$ We will start our investigation with the following lemma.
\begin{lem} \label{lem:algebra_connected}
    Suppose that $R$ is an $\F_p$-algebra and ${\rm Cay}(R, (R^{\times})^p)$ is connected. Then $R$ is a finite field. 
\end{lem}

\begin{proof}
    Let $H_1$ (respectively $H_2$) be the abelian group of generated by $R^{\times}$ (respectively $(R^{\times})^p).$ We claim that 
    \[ H_2 = H_1^p = \{a^p | a \in H_1 \}. \]
    Let $x$ be an element of $H_2$. Then we can write 
    \[ x = \sum_{i=1}^d n_i r_i^p,\]
    where $n_i \in \Z$ and $r_i \in R^{\times}.$ By Fermat's little theorem, we know that $n_i^p =n_i$ for all $1 \leq i \leq d.$ Therefore, we can write 
    \[ x = \sum_{i=1}^d n_i^p r_i^p = \left(\sum_{i=1}^d n_i r_i \right)^p.\]
    We conclude that $x \in H_1^p$. Therefore, $H_2 \subseteq H_1^p.$ By a similar argument can show that $H_1^p \subseteq H_2.$ This shows that $H_2 = H_1^p.$

    Now, because ${\rm Cay}(R, (R^{\times})^p)$ is connected, we must have $H_2 = R.$ Consequently $H_1^p =R$ and hence $R = R^p.$ This implies that the Frobenius map $\Phi: R \to R$ sending $r \mapsto r^p$ is an isomorphism. Since $M$ is nilpotent, we must have $M=0.$ In other words, $R$ is a field.
\end{proof}

\begin{cor} \label{cor:connected}
    Let $R$ be a finite local commutative  ring such that $k=R/M$ has characteristics $p.$ Suppose that ${\rm Cay}(R, (R^{\times})^p)$ is connected. Then $M=pR$ and $R/pR$ is a finite field. 
\end{cor}

\begin{proof}
Apply \cref{lem:algebra_connected} for the ring $R/p.$
    \end{proof}

\begin{prop} \label{prop:connected}
    Let $S$ be a subset of $R$. Then ${\rm Cay}(R, S)$ is connected if and only if ${\rm Cay}(R/p, \varphi(S))$ is. Here $\varphi: R \to R/p$ is the canonical map. 
\end{prop}
\begin{proof}
    Let $H$ be the abelian group generated by $S.$ By definition, $\varphi(H)$ is the abelian group generated by $\varphi(S).$ As a result, if ${\rm Cay}(R,S)$ is connected, then $ {\rm Cay}(R/p, \varphi(S))$ is connected as well. 
    
    Now, suppose that  ${\rm Cay}(R/p, \varphi(S))$ is connected. We then have $\varphi(H) =R/p$. We claim that $H=R.$ Let $r \in R.$ Because $\varphi(H)=R/p$, we can find $h_1 \in H$ such that $\phi(h_1)=\phi(r).$ This implies that $r-h_1 \in \ker(\varphi) = pR.$ Therefore, we can write  
    \[ r = h_1 + pr_1,\]
    where $h_1 \in H$ and $r_1 \in R.$ Keeping the same process, we see that for each $n \geq 1,$ we can find $(h_1, h_2, \ldots, h_n) \in H^n$ and  $r_n \in R$ such that 
    \[ r = h_1 + ph_2 +\cdots+ p^{n-1} h_n +p^n r_n. \]
    Because $p$ is nilpotent, we can find $n \in \N$ such that $p^n=0.$ Consequently, in the above equation, we would have 
    \[ r = h_1 + ph_2 +\cdots+ p^{n-1} h_n \in H.
    \qedhere\]
\end{proof}

Let $(R,M)$ be a local ring and $\phi: R \to R/M:=k$ be the canonical map. 

\begin{prop} \label{prop:unit_power}
    Let $(R,M)$ be a finite commutative local ring such that $k=R/M$ has characteristics $p.$ Let $a \in k^{\times}.$ Then 
    \begin{enumerate}
        \item There exists $x \in R^{\times}$ such that $\phi(x^p)=a.$
        \item Suppose that $M=pR$. If $x_1, x_2 \in R^{\times}$ such that $\phi(x_1^p)=\phi(x_2^p)=a$, then $x_1^p \equiv x_2^p \pmod{p^2R}.$
    \end{enumerate}
    \end{prop}
\begin{proof}
Since $k^{\times}$ is a cyclic of order prime to $p$, we have $k^{\times}=(k^{\times})^p.$ Therefore, we can find $b \in k^{\times}$ such that $a=b^p.$ Let $x$ be any lift of $b$ to $R$; namely $\phi(x)=b$. We then see that  $x\in R^\times$ and
\[ \phi(x^p) = \phi(x)^p=b^p=a. \]
For the second part, we observe that we have 
\[ 0 = (\phi(x_1)^p - \phi(x_2)^p) = (\phi(x_1)-\phi(x_2))^p.\]
Because $R/pR$ is a field, we must have $\phi(x_1)=\phi(x_2).$ In other words, $x_1 = x_2+pa$ for some $a \in R.$ We can then see that $x_1^p \equiv x_2^p \pmod{p^2R}$. 
\end{proof}
\begin{cor} \label{cor:size_of_R}
    Suppose that $M=pR$ and $p^2R=0.$ Then, for each $a \in k^{\times}$, there exists a unique $y \in (R^{\times})^p$ such that $\phi(y)=a.$ In other words, the induced map $\phi: (R^{\times})^p \to k^{\times} = (k^{\times})^p$ is an isomorphism. 
\end{cor}

By \cref{prop:connected} and \cref{prop:unit_power}, we have the following.

\begin{prop} \label{prop:local_connected}
    Let $(R,M)$ be a finite commutative local ring such that $k=R/M$ has characteristics $p.$ Then ${\rm Cay}(R, (R^{\times})^p)$ is connected if and only $M = pR$. Furthermore, in the case $p=2$, ${\rm Cay}(R, (R^{\times})^p)$ is connected if and only if $M = pR=0$ if and only if $R$ is a field if and only if $G_R(p) \cong K_{|R|}.$
\end{prop}
\begin{proof}
    The forward direction follows from \cref{cor:connected}. Now suppose that $M=pR.$ By \cref{prop:connected}, $G_R(p)$ is connected if and only if $G_{R/p}(p)= G_{k}(p)$ is also connect where $k=R/M$. Since $k=R/M$ is a finite field of characteristics $p$, $|k^{\times}|$ is a group with order prime to $p.$ As a result, $k^{\times} = (k^{\times})^p = k \setminus \{0 \}$ and hence $G_{k}(p) = K_{|k|}$ the complete graph on $|k|$ nodes. In particular, it is connected. We conclude that $G_R(p)$ is connected as well.

    Let us consider the case $p=2$. Because $G_R(p)$ is connected, by the above argument, we know that $M=2R$ and $R/M$ is a field of characteristics $p=2.$ We claim that $M=0$.  In fact, by our assumption that $G_{R}(2)$ is an undirected graph, $-1 \in (R^{\times})^2$; i.e, there exists $x \in R$ such that $x^2=-1.$ Let $\bar{x} \in R/M = R/2R$ be the projection of $x$ to the residue field of $R$. Then $\bar{x}^2=\overline{-1} = \bar{1}.$ Consequently $(\bar{x}-1)^2=0$. Since $R/2R$ is a field, we must have $\bar{x}=1.$ In other words, we can write $x = 1+2a$ for some $a \in R.$ We then have 
     \[ 0 = x^2+1 = (1+2a)^2+1 = 2(1+2a+2a^2).\]
     Since $1+2a+2a^2 \in R^{\times}$, we conclude that $2=0$ in $R.$ This shows that $M=2R=0$ and $R$ is a field of characteristics $2.$ In this case $R^{\times} = (R^{\times})^2 = R \setminus \{0 \}.$  Consequently $G_R(p) = K_{|R|}.$
\end{proof}

\begin{cor} \label{cor:p_p_bipartite}
    $G_R(p)$ is a connected bipartite undirected graph if and only if $p=2$ and $R=\F_2.$
\end{cor}

\begin{proof}
     $G_R(p)$ contains the $p$-cycle 
     \[ 0 \to 1 \to \cdots \to p-1 \to 0.\]
     Consequently, if $p$ is odd, then $G_R(p)$ contains an odd cycle and therefore it is not bipartite. Let us consider the case $p=2$. By \cref{prop:local_connected}, we know that $G_R(p) = K_{|R|}$ in this case. Therefore, $G_R(p)$ is bipartite if and only $|R|=2.$ In other words, $R=\F_2.$
\end{proof}

\begin{prop} \label{prop:local_anti_connected}
    Let $(R,M)$ be a finite commutative local ring such that $k=R/M$ has characteristics $p.$ Suppose that $G_R(p)$ is connected. Then ${\rm Cay}(R, (R^{\times})^p)$ is anti-connected if and only if $R$ is not a field. 
\end{prop}
\begin{proof}
   If $R$ is a field then $G_R(p) = K_{|R|}.$ As a result, the complement of $G_R(p)$ is the empty graph $E_{|R|}$ and hence,$G_R(p)$ is not anticonnected. Conversely, suppose that $R$ is not a field. We claim that $G_R(p)$ is anticonnected. First of all, since $G_R(p)$ is connected, \cref{prop:connected} implies that $M=pR.$ We claim that $\varphi(R \setminus (R^{\times})^p) = k$ where $k = R/M=R/pR$ and $\varphi\colon R \to R/p$ is the canonical projection map. Let $\bar{x} \in k^{\times}$. By \cref{prop:unit_power}, we can find $y \in R^{\times}$ such that $\varphi(y^p)=\bar{x}.$ We claim that $y^p+a \not \in (R^{\times})^p$ for each $a \in pR \setminus p^2 R.$ We remark that, by definition $\varphi(y^p+a)=\bar{x}$, and $pR \setminus p^2 R\not=\emptyset$ (if $pR=p^2R$ then $p=p^2c$ for some $c\i R$, hence $p=0$ since $1-pc\in R^\times$, a contraction). Suppose to the contrary that $y^p+a = z^p$ for some $z \in R.$  Over $R/p$, we then have 
   \[ 0 = \bar{z}^p - \bar{y}^p = (\bar{z}-\bar{y})^p.\]
   Because $R/p$ is a field, we conclude that $\bar{z}=\bar{y}.$ In other words, we can write $z=y+pt$ for some $t \in R.$ We then have 
   \[ a = z^p - y^p = (y+pt)^p - y^p = p^2w\] 
   for some $w \in R.$ Because $a \in pR \setminus p^2R$, we can write $a=pb$ for some $b \in R^{\times}$. We then have $p(b-pw)=0$. Because $b-pw \in R^{\times}$, we conclude that $p=0$ and hence $pR=0.$ This is a contradiction since we assume that $R$ is not a field. 
\end{proof}

By \cref{prop:local_connected} and \cref{prop:local_anti_connected}, we have a complete answer to \cref{question:connected}. We now focus on \cref{question:prime}. For this question, we have the following key observation which is discovered through various experiments with Sagemath. 
\begin{prop} \label{prop:p^2}
Let $(R,M)$ be a finite commutative local ring such that $k=R/M$ has characteristics $p.$ Then $p^2R$ is a homogeneous set in ${\rm Cay}(R, (R^{\times})^p)$.     
\end{prop}
\begin{proof}
  By \cite[Proposition 4.4]{chudnovsky2024prime}, for $p^2R$ to be a homogeneous set in $G_R(p)$  we need to show that 
    \[ p^2 R + (R^{\times})^p \subseteq (R^{\times})^p.\]
Let $x \in (R^{\times})$ and $a \in R$, we claim that $x^p+p^2a \in (R^{\times})^p.$ From the equation
\[ x^p + p^2 a = x^p(1+p^2 (ax^{-p})),\]
we can assume that $x=1.$ While Hensel's lemma does not apply directly, we can modify it to fit our situation. Specifically, we claim that there for each $n \geq 1$, we can find $x_n \in R$ such that 
\[ x_n^p \equiv 1+p^2a \pmod{p^{n+1}}. \]
We remark the above congruence immediately implies that $x_n \in R^{\times}$. For $n=1$, we can take $x_1 = 1+ap$. In this case 
\[ x_1^p = (1+ap)^p = 1 + ap^2 + \sum_{k=2}^p {p \choose k } (ap)^k \equiv 1 +ap^2 \pmod{p^2}. \]
Suppose that the statement has been verified for all $n$. We claim that it is also true for $n+1.$ In fact, let $x_{n+1} = x_n + p^n b$ for some $b \in R.$ We then have 
\[ x_{n+1}^p = (x_n+p^n b)^p \equiv x_n^p+x_n^{p-1} p^{n+1} b  \equiv 1+p^2 a + (x_n^p-1-p^2a+x_n^{p-1} p^{n+1} b) \pmod{p^{n+2}}. \]
By the induction hypothesis, we can write $x_n^p-1-p^2a = p^{n+1}c_n$ for some $c_n \in R.$ If we take $b = -(x_n^{-1})^{p-1} c_n$ then 
\[ x_{n+1}^p \equiv 1 +p^2 a \pmod{p^{n+2}}.\]
Since $p$ is nilpotent in $R$, $p^n=0$ for some large enough $n.$ Therefore, we can find $n$ such that $x_{n}^p = 1+p^2a.$ This shows that $1+p^2a \in (R^{\times})^p.$
\end{proof}

By \cite[Corollary 4.2]{chudnovsky2024prime} and \cref{cor:size_of_R}, we have the following. 
\begin{cor}
\label{cor:wreathproduct}
For each $m \geq 1$, we define 
\[ R_m = R/p^m.\]
Then $G_R(p)$ is the wreath product of $G_{R_2}(p)$ and $E_n$ where $n=|p^2 R|$; i.e. (see \cref{def:wreath_product} for the definition of the wreath product of two graphs)
 \[ G_R(p) \cong G_{R_2}(p) \cdot E_{n}.\]
 Furthermore, $G_{R_2}(p)$ is a regular graph of degree $|R_1|-1.$
\end{cor}

\begin{expl}
    Figure \cref{fig:z_25} shows the graph $G_{\Z/25}(5).$ It is a regular graph of degree $4.$
\begin{figure}[H]
\centering
\includegraphics[width=0.4 \linewidth]{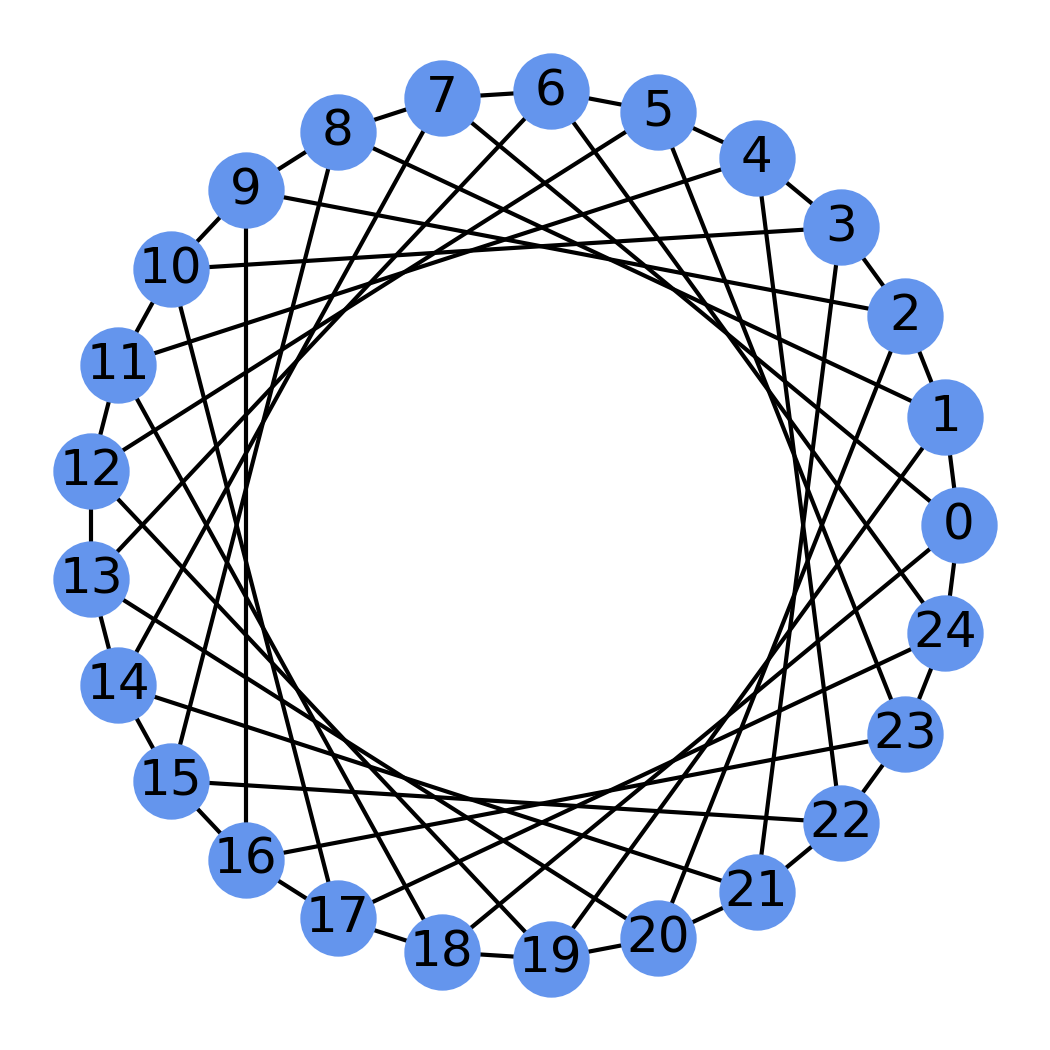}
\caption{The Cayley graph $G_{\Z/25}(5)$}
\label{fig:z_25}
\end{figure}
\end{expl}

We now show that $p^2R$, under mild conditions, is the maximal homogeneous set in $G_R(p).$
\begin{prop} \label{prop:condition}
Let $(R,M)$ be a finite commutative local ring such that $k=R/M$ has characteristics $p.$ Let $a \in M^2 \setminus M.$ Then 
\[ (a+ (R^{\times})^p) \cap (R^{\times})^p = \emptyset .\]
\end{prop}
\begin{proof}
The key argument for the proof of this proposition is somewhat similar to the one given in \cref{prop:local_anti_connected}. For the sake of completeness, we provide it here for the reader's convenience.  Let us assume to the contrary that $(a+ (R^{\times})^p) \cap (R^{\times})^p \neq \emptyset$. Then we can find $x, y \in R^{\times}$ such that\[ a+x^p = y^p.\] By projecting this equation over the residue field $R/M$ we see that $(\bar{x}-\bar{y})^p=0.$ Consequently, $\bar{x}=\bar{y}.$ In other words, we can write $y = x + m$ for some $m \in M.$ We then have 
    \[ a= y^p-x^p = (x+m)^p - x^p = pm \left(\sum_{k=1}^{p-1} \dfrac{{p \choose k}}{p} x^k m^{p-k-1} \right)+m^p \in M^2.\]
    This contradicts our assumption that $a \not \in M^2.$
\end{proof}
We have the following corollary. 
\begin{cor} \label{cor:p^2_max}
    Suppose that $I$ is a proper ideal in $R$ and $I$ is a homogeneous set in $G_R(p).$ Suppose further that $G_R(p)$ is connected. Then $I \subseteq p^2 R$. In other words, $p^2R$
 is the maximal ideal which is also a homogeneous set in $G_R(p).$
 \end{cor}
 \begin{proof}
     Because $I$ is a homogeneous set, by \cite[Proposition 4.4]{chudnovsky2024prime} we know that 
     \[ I + (R^{\times})^p \subseteq (R^{\times})^p. \]
     By \cref{prop:condition}, we know that for each $a \in I$, $a \in M^2.$ Because $G_R(p)$ is connected, by \cref{prop:local_connected}, we know that $M=pR.$ As a result, $a \in p^2R.$ Since this is true for all $a \in I$, we conclude that $I \subseteq p^2R.$
 \end{proof}
\begin{rem}
    We remark that we do not require $G_R(p)$ to be anticonnected in the proof of \cref{cor:p^2_max}.
\end{rem}
\begin{thm} \label{thm:prime_local_p_p}
Let $(R,M)$ be a finite commutative local ring such that $k=R/M$ has characteristics $p.$ Then ${\rm Cay}(R, (R^{\times})^p)$ is a prime graph if and only if the following conditions hold. 
\begin{enumerate}
    \item $R$ is not a field.
    \item $M=pR$. 
    \item $p^2R =0.$
\end{enumerate}

\begin{proof}
    First, let us assume that $G_R(p)$ is a prime graph. Then $G_R(p)$ must be connected and anticonnected. By \cref{prop:local_connected}, and \cref{prop:local_anti_connected} we conclude that $M=pR$ and $R$ is not a field. Furthermore, by \cref{prop:p^2}, $p^2R$ is a homogeneous set in $G_R(p).$ Because $G_R(p)$ is prime, we must have $p^2R=0.$

    Conversely, suppose that the three conditions above are satisfied. We claim that $G_R(p)$ is prime. First, \cref{prop:local_connected}, and \cref{prop:local_anti_connected} shows that $G_R(p)$ is connected and anticonnected. Suppose that $G_R(p)$ is not prime. 
    By \cref{prop:prime_old}, we know that there exists a proper ideal $I$ in $R$ such that $I \neq 0$ and $I$ is a homogeneous set in $G_R(p).$ By \cref{cor:p^2_max}, we know that $I \subseteq p^2R.$ By our assumption $p^2R=0$ and therefore $I=0.$ This is a contradiction. We conclude that $G_R(p)$ is a prime graph. 
\end{proof}
\end{thm}

\begin{rem}
    A particular example of rings that satisfy the conditions of \cref{thm:prime_local_p_p} is the class of Galois rings (see \cite[Section 6.1]{bini2012finite})
    \[ R = GR(p^2, r)= \Z[x]/(p^2,f(x)), \]
    where $f(x) \in \Z[x]$ is an irreducible polynomial modulo $p.$
\end{rem}

\begin{rem}
    In general, a ring that satisfies the conditions of \cref{thm:prime_local_p_p} must be a quotient ring of the Laurent series over a Cohen ring (see \cite[Theorem 10.160.8]{stacks-project}). 
\end{rem}

\subsection{Induced subgraphs of $G_R(p)$} \label{induced_p_p}

In \cref{sec:induced_l_p} we study various induced subgraphs of $G_R(p)$ where $R$ is a field of characteristics $\ell \neq p.$ In this section, we study a similar problem for $G_R(p)$ where $R$ is a local ring of residue characteristics $p.$ Since our main interest lies in the case $G_R(p)$ is connected, we will make that assumption throughout this section. By \cref{prop:connected}, this would imply that $M=pR.$ Furthermore, in the case $p=2$, $G_R(p) = K_{|R|}$ and hence our problem is rather trivial in this case. Therefore, we will also assume that $p \geq 3.$ 

\begin{lem} \label{lem:K_3_induced}
     The following conditions are equivalent. 
     \begin{enumerate}
     \item $K_3$ is an induced subgraph of $G_R(p).$ 
     \item There exists $a \in R^{\times}$ such that the induced subgraph on $\{0, 1, -a^p \}$ is $K_3.$ In other words, there exists $a \in R^{\times}$ such that $1+a^p \in (R^{\times})^p.$
     \end{enumerate}
\end{lem}

\begin{proof}
Clearly $(2)$ implies $(1)$. Let us show that $(1)$ implies $(2)$ as well.   By assumption, there exists $u_1, u_2, u_3 \in R$ such that the induced graph on $\{u_1, u_2, u_3 \}$ is $K_3.$ By definition, $u_2-u_1 \in (R^{\times})^p$ and $u_1-u_3 \in (R^{\times})^p$. Hence there exists $a\in R^\times$ such that $a= \left(\dfrac{u_1-u_3}{u_2-u_1}\right)^p$. We can then see that the induced subgraph on 
    \[ \left\{ \frac{u_1-u_1}{u_2-u_1}, \frac{u_2-u_1}{u_2-u_1}, \frac{u_3-u_1}{u_2-u_1} \right\} =\left\{0, 1, \frac{u_3-u_1}{u_2-u_1}\right\}=\left\{0,1,-a^p \right\}\]
    is also $K_3.$ 
\end{proof}

Inspired by \cref{lem:K_3_induced}, we define the following polynomial 
\[ f_p(x) = (1+x)^p-x^p-1 \in \Z[x].\]
The introduction of $f_p(x)$ is suggested to us by Dr. Ha Duy Hung. We thank him for sharing this idea. 

\begin{expl}
    Here are some examples of $f_p(x)$ for small $p.$ 
    \[ f_3(x) = 3x(x+1).\]
    \[ f_5(x) = 5 x(x + 1) (x^2 + x + 1) ,\] 
    \[ f_7(x) = 7  x (x + 1) (x^2 + x + 1)^2 ,\] 
    \[ f_{11}(x) = 11 x (x + 1) (x^2 + x + 1) (x^6 + 3x^5 + 7x^4 + 9x^3 + 7x^2 + 3x + 1).\]
    
\end{expl}

\begin{prop} \label{prop:root_f_p}
    $K_3$ is an induced subgraph of $G_R(p)$ if and only there exists $a \in R^{\times}$ such that $a+1 \in R^{\times}$ and $f_p(a)=0.$
\end{prop}
\begin{proof}
    If $a \in R^{\times}$ such that $1+a \in R^{\times}$ and $f_p(a)=0$ then the induced graph on $\{0, 1, -a^p \}$ is $K_3.$ Converesly, suppose that $K_3$ is an induced subgraph of $G_R(p).$ By \cref{lem:K_3_induced}, we can find $a \in R^{\times}$ such that 
    \[ 1+ a^p = b^p,\]
    for some $b \in R^{\times}.$ By \cref{cor:wreathproduct},
    we can assume that $p^2R=0$. By taking this equation over $R/pR$, we conclude that $b \equiv 1 +a \pmod{p}.$ Consequently 
    \[ b^p \equiv (1+a)^p \pmod{p^2R}.\]
    Since $p^2R=0$, we conclude that $b^p=(1+a)^p$ and that $f_p(a)=0.$ We remark that since $b \equiv 1+a \pmod{pR}$ and $b \in R^{\times}$, $1+a \in R^{\times}$ as well. 
    \end{proof}
Let us next discuss a particular factor of $f_p(x).$ 

\begin{lem} \label{lem:phi_3}
    Suppose that $p>3.$ Let $m$ be the multiplicity of $x^2+x+1$ in $f_p(x).$ Then 

    \[
 m= 
\begin{cases}
    1 & \text{if } p \equiv 2 \pmod{3} \\
    2 & \text{else.} 
\end{cases}
\]
\end{lem}

\begin{proof}
    Let $\zeta_3$ be a primitive $3$-root of unity. Then the minimal polynomial of $\zeta_3$ is $\Phi_3(x) = x^2+x+1$ and $m$ is the multiplicity of $\zeta_3$ in $f_p(x).$ Therefore, in order to calculate $m$, we need to study $f_{p}^{(k)}(\zeta_3)$ for $k \geq 0.$ First, we observe that  
    \[ \zeta_3+1 = -\zeta_3^2.\]
    Consequently 
    \[ f_p(\zeta_3) = (-\zeta_3^2)^p-\zeta_3-1 =-(\zeta_3^{2p}+\zeta_3^p+1) = \zeta_3^2+\zeta_3+1=0.\]
    Similarly, we have 
    \[ f_p'(\zeta_3) = p [(-\zeta_3^2)^{p-1} - \zeta_3^{p-1}] = p \zeta_{3}^{p-1}(\zeta_{3}^{p-1}-1).\]
    We can see that $f_p'(\zeta_3)=0$ if and only if $p \equiv 1 \pmod{3}.$ Finally, assume now that $p \equiv 1 \pmod{3}$ then 
    \[ f_p''(\zeta_3) = p(p-1) [(-\zeta_3^2)^{p-2} - \zeta_3^{p-2}] = p (p-1) \neq 0.
    \qedhere\]
\end{proof}

\begin{cor} \label{cor:factor_f_p}
   Suppose that $p>3$. There exists a polynomial $g_p(x) \in \Z[x]$ such that 
    \[ f_p(x) = p x(x+1) (x^2+x+1)^m g_p(x),\]
    with $g(\zeta_3) \neq 0.$ Here 
        \[
 m= 
\begin{cases}
    1 & \text{if } p \equiv 2 \pmod{3} \\
    2 & \text{if } p \equiv 1 \pmod{3}. 
\end{cases}
\]
\end{cor}

\begin{rem}
    We have verified that $g_p(x)$ is irreducible for all $p <10^4.$ We wonder whether this is the case for all $p$. 
\end{rem}
We have the following partial observation. 
\begin{lem}
If $a \in \C$ is a repeated root of $f_p(x)$, then $a^2+a+1=0.$ Consequently, $g_p(x)$ is separable over $\Z[x].$ 
\end{lem}

\begin{proof}
    Because $a$ is a repeated root of $f_p(x)$, we have $f_p(a)=f_p'(a)=0.$ We have 
    \[ 0 = f_p'(a) = p \left((a+1)^{p-1}-a^{p-1} \right)=0.\]
    There exists $\lambda \in \C$ such that $\lambda^{p-1}=1$ and $a+1= \lambda a.$ We then see that $a = \frac{1}{\lambda-1}.$ Substituting this into the equation $f_p(a)=0$, we have 
    \[  \left(\frac{\lambda}{\lambda-1} \right)^p - \left(\frac{1}{\lambda-1} \right)^p-1 =0.\]
    Using the fact that $\lambda^{p-1}=1$, we see that 
    \[ 0 = \lambda^p - (\lambda-1)^p -1 = \lambda - (\lambda-1)^p-1 = (\lambda-1)(1-(\lambda-1)^{p-1}).\]
    We conclude that $(\lambda-1)^{p-1}= \lambda^{p-1}=1.$ We can find $0 \leq k, l \leq p-1$ such that
    \[ \lambda = e^{\frac{2 k}{p-1}i} = \cos\left(\frac{2k}{p-1} \right)+i \sin \left(\frac{2k}{p-1} \right),\]
    \[ \lambda -1 = e^{\frac{2 l}{p-1}i} = \cos\left(\frac{2l}{p-1} \right)+i \sin \left(\frac{2l}{p-1} \right). \]
    We conclude that 
    \[\sin \left(\frac{2k}{p-1} \right) = \sin \left(\frac{2l}{p-1} \right), \cos\left(\frac{2k}{p-1} \right) = 1+\cos\left(\frac{2l}{p-1} \right). \]
    From this, we can see that 
\[\cos\left(\frac{2k}{p-1} \right) = \frac{1}{2}, \cos\left(\frac{2l}{p-1} \right)= -\frac{1}{2}. \]
We conclude either $\lambda = \frac{1 + \sqrt{3}i}{2}$ or $\lambda = \frac{1-\sqrt{3}i}{2}.$ We then have $a = \pm\zeta_3$; i.e, $a^2+a+1=0.$
By \cref{cor:factor_f_p}, $g_p(a)\neq 0$.  We conclude that $g_p(x)$ has no repeated roots. 
\end{proof}
We now show that there is an analogous statement over $\F_p[x].$
\begin{prop} \label{prop:factor_mod_p}
    Let $h_p(x)=\dfrac{1}{p}f_p(x) = (x^2+x+1)^m g_p(x)\in \Z[x]$ where $m$ is given in \cref{lem:phi_3}. The following conditions are equivalent. 
    \begin{enumerate}
        \item   $a \in \bar{\F}_p$ is a repeated root of $h_p(x)$.
        \item  $a$ is a root of $h_p(x)$ and $ a \in \F_p\setminus\{0,-1\}$. 
    \end{enumerate} 
    Furthermore, in this case, the multiplicity of $a$ is exactly $2.$
\end{prop}
\begin{proof} Clearly $h'_p(x)=(x+1)^{p-1}-x^{p-1}$.
Suppose that $h_p$ has a root $a\in \F_p\setminus\{0,-1\}$. Then $h'_p(a)=(a+1)^{p-1}-a^{p-1}=1-1=0$.

Now we suppose that $h_p$ has a multiple root $a\in \bar{\F}_p$. From $0=h'_p(a)=(a+1)^{p-1}-a^{p-1}$, we see that $a\not=0$, $a\not=-1$ and $\left(\dfrac{a+1}{a}\right)^{p-1}=1$. Hence $\lambda:=\dfrac{a+1}{a}\in \F_p$ and thus $a=\dfrac{1}{\lambda-1}\in \F_p$. 

Finally, we note that 
\[ h_p''(a) = (p-1) \left((a+1)^{p-2} - a^{p-2}) \right) = (p-1) \left(\frac{1}{a+1}-\frac{1}{a} \right) = -\frac{p-1}{a(a+1)} \neq 0. \] 
Therefore, the multiplicity of $a$ is $2.$
\end{proof}

\begin{prop} \label{prop:p_1_mod_3}
    Suppose that $p \equiv 1 \pmod{3}$. Then $K_3$ is an induced subgraph of $G_R(p).$
\end{prop}
\begin{proof}
    Let $k=R/pR$, the residue field of $R.$  Since $p \equiv 1 \pmod{3}$, $k^{\times}$ is a cyclic group of order divisible by $3.$ Consequently, we can find $a_0 \in k^{\times}$ such that $a_0 \neq 1$ and $a_0^3=1.$ In other words, $a_0^2+a_0+1=0.$ By Hensel's lemma, we can lift $a_0$ to a root $a \in R$; namely $a^2+a+1=0$ and $\bar{a}=a_0$ where $\bar{x}$ is the projection of $x$ to $k.$ We remark that since $1+\bar{a}=1+a_0 \neq 0$, $1+a \in R^{\times}.$ Furthermore, 
    by \cref{lem:phi_3} we have $f_p(a)=0$. By \cref{prop:root_f_p}, we conclude that the induced graph on $\{0, 1, -a^p \}$ is $K_3.$
\end{proof}
\begin{rem}
    The converse of \cref{prop:p_1_mod_3} is not true. The smallest counterexample is $p=59.$ In this case, we can check that $a=4$ is a solution of $g_p(x)=0$ over $\Z/p$ and hence of $f_p(x)=0$ over $\Z/p^2.$ For prime $p<500$, $g_p(x)$ has a root in $\F_p$ if and only if 
    $p \in \{59, 79, 83, 179, 193, 227, 337, 419, 421, 443, 457 \}.$ It seems that there exists infinitely many $p$ such that $g_p(x)$ has a root in $\F_p.$ 
\end{rem}

\section{The general case} \label{sec:general_case}
In \cref{sec:local_ring_l_p} and \cref{sec:local_ring_p_p}, we provide complete answers for \cref{question:connected} and \cref{question:prime}. In this section, we study these questions in the general case; namely 
\begin{equation} \label{eq:decomposition} R = \prod_{i=1}^d R_i. 
\end{equation}
where $R_i$'s are finite local rings whose maximal ideals are $M_i$'s.  For the rest of this section, we will fix the decomposition given in \cref{eq:decomposition}. With this decomposition, we have an isomorphism 
\[ G_R(p) \cong G_{R_1}(p) \times G_{R_2}(p) \times \cdots \times G_{R_d}(p).\]

We first deal with \cref{question:connected}. By Weichsel’s Theorem (see \cite[Theorem 5.9 and Corollary 5.10]{hammack2011handbook}, we have the following. 

\begin{prop} \label{thm:weichsel}
    $G_R(p)$ is connected if and only if the following conditions hold 
    \begin{enumerate}
        \item $G_{R_i}(p)$ is connected for each $1 \leq i \leq d$,
        \item At most one of the $G_{R_i}(p)$ is a bipartite graph. 
    \end{enumerate}
\end{prop}

When $R$ is a local ring, via \cref{prop:l_p_connected} (for the case  $p \in R^{\times}$) and \cref{prop:local_connected} (for the case $p \not \in R^{\times}$) we provide the explicit conditions for $G_R(p)$ to be connected. In this case, we also provide some sufficient conditions for $G_R(p)$ to be non-bipartite (see \cref{cor:contains_cycle} and \cref{cor:p_p_bipartite}). For the general case, we have the following observation. 

\begin{prop}
    Suppose that $G_1, G_2$ are graphs.
    Then the directed product $G_1 \times G_2$ is always anti-connected. 
\end{prop}

\begin{proof}
    Let $(u_1, u_2)$ and $(v_1, v_2)$ be two vertices in $G_1 \times G_2.$ Then we have the following path in $(G_1 \times G_2)^c$
    \[ (u_1, u_2) \to (u_1, v_2) \to (u_2, v_2). \]
\end{proof}

\begin{cor} \label{cor:connected=>anti}
    Suppose that $d \geq 2$. Then $G_R(p)$ is anticonnected. 
\end{cor}

We now return to \cref{question:prime}. Under the decomposition $R = \prod_{i=1}^d R_i$, we know that each ideal $I$ is $R$ is of the form 
\[ I = I_1 \times I_2 \times \cdots \times I_d, \]
where $I_i$ is an ideal in $R_i$ for each $1 \leq i \leq d.$ We have the following lemma. 
\begin{lem} \label{lem:I_homo}
    Suppose that $I = I_1 \times I_2 \times \cdots \times I_d $ is a proper ideal in $R$; i.e., $I \neq R.$ Then $I$ is a homogeneous set in $R$ if and only if the following conditions hold
    \begin{enumerate}
        \item $I_i$ is a homogeneous set in $R_i$ for each $1 \leq i \leq d$, 
        \item $I_i \neq R_i$ for each $1 \leq i \leq d.$
        \end{enumerate}
\end{lem}

\begin{proof}
    This statement follows from the fact that $I$ is a homogeneous set in $R$ if and only if 
\[ I + (R^{\times})^p \subseteq (R^{\times})^p. \]
Under the decomposition given in \cref{eq:decomposition} this is equivalent to 
\[ I_i + (R_i^{\times})^p \subseteq (R^{\times})^p, \forall 1 \leq i \leq d. \]
This condition is equivalent to the conditions $(1)+(2)$ above. 
\end{proof}

\begin{thm} \label{thm:prime_general_case}
   Suppose that $d \geq 2$.  $G_{R}(p)$ is prime if and only if the following conditions hold 
\begin{enumerate}
    \item $G_{R_i}(p)$ is a connected graph for each $1 \leq i \leq d$,
    \item $G_{R}(p)$ is connected,
    \item If $p$ is invertible in $R_i$ then $R_i$ is a field, 
    \item If $p$ is not invertible in $R_i$ then $M_i=pR_i$ and $p^2 R_i=0$. Here $M_i$ is the maximal ideal in $R_i.$
\end{enumerate}
\end{thm}

\begin{proof}
    Suppose that $G_R(p)$ is prime. Then $(1), (2)$ hold. We note that if $p$ is not invertible in $R$ and $G_{R_i}(p)$ is connected, we must have $M_i =p R_i$ (by \cref{cor:connected}). For each $1 \leq i \leq d$, we define 

    \[
 J_i = 
\begin{cases}
    M_i & \text{if $p \in R_i^{\times}$}  \\
    p^2 M_i & \text{else.} 
\end{cases}
\]
Then $J = \prod_{i=1}^d J_i$ is a homogeneous set in $G_R(p)$. Because $G_R(p)$ is prime, we must have $J= 0$ or equivalently $J_i=0$ for each $1 \leq i \leq d.$

Conversely, suppose that all conditions are satisfied. We claim that $G_R(p)$ is prime. Suppose to the contrary that $G_R(p)$ is not prime. By \cref{cor:connected=>anti}, we know that $G_R(p)$ is also anticonnected. By \cite[Theorem 4.1]{chudnovsky2024prime}, there exists a non-zero proper ideal $I$ in $R$ such that $I$ is a homogeneous set in $R.$ Let us write $I= \prod_{i=1}^d I_i$. By \cref{lem:I_homo}, we know that each $I_i$ is a homogeneous set in $R_i.$ By the definition of $J_i$ together with \cref{cor:p^2_max}, we know that $I_i \subseteq J_i$ for each $1 \leq i \leq d.$ By our assumption, we must have $I_i = \{0 \}$ and hence $I = 0.$ This is a contradiction. 
\end{proof}

\begin{rem}
    We note that condition $(1)$ is described explicitly in \cref{prop:l_p_connected} (for the case $p \in R_i^{\times})$ and in \cref{prop:local_connected} (for the case $p \not \in R^{\times}.)$
\end{rem}
We provide some partial results where the conditions in \cref{thm:prime_general_case} can be described more explicitly. First, by Weichsel's theorem \cref{thm:weichsel} and the fact that $G_{R_i}(p)$ is not bipartite if $2  \in R_i^{\times}$ (see \cref{cor:contains_cycle}) we have the following. 
\begin{prop}
    Suppose that there exists at most one $i$ such that $2 \not \in R_i^{\times}.$ Then $G_R(p)$ is connected if and only $G_{R_i}(p)$ is connected for all $1 \leq i \leq d.$
\end{prop}

When $p=2$, we have the following. 
\begin{thm}
   Suppose that $d \geq 2.$ Then $G_R(2)$ is prime if and only if the following conditions hold. 
\begin{enumerate}
\item There exists at most one $1 \leq i \leq d$ such that $R_i = \F_2$.
    \item If $2$ is not invertible in $R_i$ then $R_i$ is a field of characteristics $2.$
    \item If $2$ is invertible in $R_i$ then $R_i$ is a field of characteristics $\ell \neq 2.$
\end{enumerate}
    
\end{thm}
\begin{proof}
   The necessary conditions follow from \cref{thm:prime_general_case}. We now show that they are sufficient as well. If $2$ is invertible in $R_i$ then $G_{R_i}(2)$ is connected and not bipartite by \cref{cor:connected} and \cref{cor:contains_cycle}. On the other hand, if $2$ is not invertible in $R_i$ then $G_{R_i}(2) \cong K_{|R_i|}$ by \cref{prop:local_connected}. In particular, it is bipartite if and only $|R_i|=2.$ We can then see that under the above conditions, all conditions mentioned in \cref{thm:prime_general_case} are satisfied. Therefore, $G_R(p)$ is prime. 
   \end{proof}

\section*{Acknowledgements}
We thank Dr. Ha Duy Hung for some correspondences around the arithmetics of the polynomial $f_p.$
\bibliographystyle{amsplain}
\bibliography{references.bib}

\providecommand{\bysame}{\leavevmode\hbox to3em{\hrulefill}\thinspace}
\providecommand{\MR}{\relax\ifhmode\unskip\space\fi MR }
\providecommand{\MRhref}[2]{%
  \href{http://www.ams.org/mathscinet-getitem?mr=#1}{#2}
}
\providecommand{\href}[2]{#2}
\begin{thebibliography}{10}

\bibitem{unitary}
Reza Akhtar, Megan Boggess, Tiffany Jackson-Henderson, Isidora Jim{\'e}nez, Rachel Karpman, Amanda Kinzel, and Dan Pritikin, \emph{On the unitary {Cayley} graph of a finite ring}, Electron. J. Combin. \textbf{16} (2009), no.~1, Research Paper 117, 13 pages.

\bibitem{andre1948courbes}
WEIL Andr{\'e}, \emph{Sur les courbes alg{\'e}briques et les vari{\'e}t{\'e}s qui s’ en d{\'e}duisent}, no. 1041, Actualit{\'e}s Sci. Ind, 1948.

\bibitem{bavsic2015polynomials}
Milan Ba{\v{s}}i{\'c} and Aleksandar Ili{\'c}, \emph{Polynomials of unitary {Cayley} graphs}, Filomat \textbf{29} (2015), no.~9, 2079--2086.

\bibitem{bini2012finite}
Gilberto Bini and Flaminio Flamini, \emph{Finite commutative rings and their applications}, vol. 680, Springer Science \& Business Media, 2012.

\bibitem{boccaletti2014structure}
Stefano Boccaletti, Ginestra Bianconi, Regino Criado, Charo~I Del~Genio, Jes{\'u}s G{\'o}mez-Gardenes, Miguel Romance, Irene Sendina-Nadal, Zhen Wang, and Massimiliano Zanin, \emph{The structure and dynamics of multilayer networks}, Physics reports \textbf{544} (2014), no.~1, 1--122.

\bibitem{bollobas1981graphs}
B{\'e}la Bollob{\'a}s and Andrew Thomason, \emph{Graphs which contain all small graphs}, European Journal of Combinatorics \textbf{2} (1981), no.~1, 13--15.

\bibitem{chudnovsky2024prime}
Maria Chudnovsky, Michal Cizek, Logan Crew, J{\'a}n Min{\'a}{\v{c}}, Tung~T Nguyen, Sophie Spirkl, and Nguy{\^e}n~Duy T{\^a}n, \emph{On prime {Cayley} graphs}, arXiv preprint arXiv:2401.06062 (2024).

\bibitem{cohen1988clique}
Stephen~D Cohen, \emph{Clique numbers of {} graphs}, Quaestiones Mathematicae \textbf{11} (1988), no.~2, 225--231.

\bibitem{ghinelli2011codes}
Dina Ghinelli and Jennifer~D Key, \emph{Codes from incidence matrices and line graphs of {Paley} graphs}, Advances in Mathematics of Communications \textbf{5} (2011), no.~1, 93.

\bibitem{hammack2011handbook}
Richard Hammack, Wilfried Imrich, and Sandi Klav{\v{z}}ar, \emph{Handbook of product graphs}, CRC press, 2011.

\bibitem{jain2023composed}
Priya~B Jain, Tung~T Nguyen, J{\'a}n Min{\'a}{\v{c}}, Lyle~E Muller, and Roberto~C Budzinski, \emph{Composed solutions of synchronized patterns in multiplex networks of kuramoto oscillators}, Chaos: An Interdisciplinary Journal of Nonlinear Science \textbf{33} (2023), no.~10.

\bibitem{javelle2014cryptographie}
J{\'e}r{\^o}me Javelle, \emph{Cryptographie quantique: Protocoles et graphes}, Ph.D. thesis, Universit{\'e} de Grenoble, 2014.

\bibitem{kivela2014multilayer}
M.~Kivel{\"a}, A.~Arenas, M.~Barthelemy, J.~P. Gleeson, Y.~Moreno, and M.~A. Porter, \emph{Multilayer networks}, Journal of Complex Networks \textbf{2} (2014), no.~3, 203--271.

\bibitem{klotz2007some}
Walter Klotz and Torsten Sander, \emph{Some properties of unitary {Cayley} graphs}, The electronic journal of combinatorics (2007), R45--R45.

\bibitem{mauduit1997finite}
Christian Mauduit and Andr{\'a}s S{\'a}rk{\"o}zy, \emph{On finite pseudorandom binary sequences i: Measure of pseudorandomness, the legendre symbol}, Acta Arithmetica \textbf{82} (1997), no.~4, 365--377.

\bibitem{paleygraph}
J{\'a}n Min{\'a}{\v{c}}, Lyle Muller, Tung~T Nguyen, and Nguyen~Duy T{\^a}n, \emph{On the {} graph of a quadratic character}, To appear in Mathematica Slovaca (2023).

\bibitem{nguyen2022broadcasting}
Tung~T Nguyen, Roberto~C Budzinski, Federico~W Pasini, Robin Delabays, J{\'a}n Min{\'a}{\v{c}}, and Lyle~E Muller, \emph{Broadcasting solutions on networked systems of phase oscillators}, Chaos, Solitons \& Fractals \textbf{168} (2023), 113166.

\bibitem{podesta2019spectral}
Ricardo~A Podest{\'a} and Denis~E Videla, \emph{Spectral properties of generalized {Paley} graphs and their associated irreducible cyclic codes}, arXiv preprint arXiv:1908.08097 (2019).

\bibitem{podesta2021_finitefield}
\bysame, \emph{The waring’s problem over finite fields through generalized {Paley} graphs}, Discrete Mathematics \textbf{344} (2021), no.~5, 112324.

\bibitem{podesta2023waring}
\bysame, \emph{Waring numbers over finite commutative local rings}, Discrete Mathematics \textbf{346} (2023), no.~10, 113567.

\bibitem{stacks-project}
The {Stacks project authors}, \emph{The stacks project}, \url{https://stacks.math.columbia.edu/tag/04GE}, 2024.

\bibitem{su2016diameter}
Huadong Su, \emph{On the diameter of unitary cayley graphs of rings}, Canadian Mathematical Bulletin \textbf{59} (2016), no.~3, 652--660.

\end{thebibliography}
\end{document}